\documentclass[letterpaper,11pt]{article}
\usepackage{amsmath}
\usepackage{amssymb}
\usepackage{amsthm}
\usepackage{latexsym}
\usepackage{graphicx}
\usepackage{setspace}
\usepackage{array}
\usepackage{booktabs}
\usepackage{ctable}
\usepackage{multirow}
\usepackage{bigstrut}
\usepackage{float}
\usepackage[top=2cm,bottom=2cm,right=2cm,left=2cm]{geometry}
\usepackage{array}
\usepackage[numbers,sort&compress]{natbib}
\newtheorem{lemma}{Lemma}
\newtheorem{theorem}{Theorem}

\newcommand{\bs}{ \boldsymbol }
\newcommand{\mb}{ \mathbf }

\begin{document}

\title{The Constrained Maximum Likelihood Estimation For Parameters Arising From Partially Identified Models}

\author{Hao Luo\footnote{corresponding author: hao.luo@stat.ubc.ca}, Alexandre Bouchard-C\^{o}t\'{e}, Gabriela Cohen Freue, Paul Gustafson \\[4pt]
University of British Columbia, Canda}
\date{}
\maketitle \abstract{We extend the constrained maximum likelihood estimation theory for parameters of a completely identified model, proposed by Aitchison and Silvey (1958), to parameters arising from a partially identified model. With a partially identified model, some parameters of the model may only be identified through constraints imposed by additional assumptions. We show that, under certain conditions, the constrained maximum likelihood estimator exists and locally maximize the likelihood function subject to constraints. We then study the asymptotic distribution of the estimator and propose a numerical algorithm for estimating parameters. We also discuss a special situation where exploiting additional assumptions does not improve estimation efficiency. \\
\ \\
\noindent Keywords: Constrained maximum likelihood estimation; Partially identified models; Lagrange multiplier method; Local maximum. \\
\ \\
\noindent MSC2010: 62F12 (Primary); 
}

\section{Introduction}

In some scientific studies, due to constraints of logistics and/or resources, data are not collected in the ideal way. Consequently, the available data may only partially identify the statistical model under consideration, i.e., parameters of the statistical model are identified up to a set of possible values instead of just one single value. The set of parameter values that correspond to the same distribution of observables is usually termed the identification region.~\citet{Manski2003} gives an overview of partial identification and covers many scenarios where partial identification may arise. 

Of course, point-identification is preferred as it is fundamental for consistent point estimation and ensures many nice properties of model-based parameter estimators. With a partially identified model, when possible one may impose some reasonable assumptions to achieve point-identification. Under such assumptions, the parameter vector will be restricted to a subset of the original parameter space. If this constrained parameter space has only a single point of intersection with the identification region, then the parameter vector is uniquely identified. 

In this paper, we study the maximum-likelihood estimation of parameters arising from a partially identified model with some equality constraints introduced by additional assumptions. In particular, we consider the scenario where there exists a special re-parameterization of all parameters of the model, which is termed a transparent re-parameterization by~\citet{GustafsonGelfandSahuEtAl2005}, such that the distribution of observables is completely determined by a proper subset of parameters after transformation. 

In the situation of adding parameter constraints to a model which is identified even without the constraints,~\citet{AitchisonSilvey1958} characterized the large-sample behavior of maximum-likelihood estimators via a Lagrange multiplier approach. However, the assumption that the unconstrained version of the model is identified is embedded in their approach. Therefore, our work extends their theory to the situation that identification is only obtained via imposition of the constraints. 

Of course, point-identification is preferred as it is fundamental for consistent point estimation and ensures many nice properties of model-based parameter estimators. With a partially identified model, when possible one may impose some reasonable assumptions to achieve point-identification. Under such assumptions, the parameter vector will be restricted to a subset of the original parameter space. If this constrained parameter space has only a single point of intersection with the identification region, then the parameter vector is uniquely identified. 

The paper is organized as follows. We first introduce some general notation and give a mathematical formulation of the problem. We then prove the existence of the constrained maximum likelihood estimate and show that the estimator is asymptotically normally distributed.  A numerical algorithm for computing the constrained maximum likelihood estimate is also developed. We then use a simulation study to compare the performance of the proposed method and the general method, which does not depend on constraints, to investigate the effect of imposing additional assumptions with a partially identified model. Moreover, we comment on a special situation where there is no benefit in terms of estimation efficiency. Finally, we present some concluding thoughts.

\section{Statistical problem}
Suppose our data consist of  $n$ observations $\mb{x} = (x_1, \dots, x_n)^{T}$. The statistical model underlying the data is assumed to be initially parameterized in scientific terms via a vector of $s$ parameters. Let $\bs{\omega} = (\omega_1, \dots, \omega_s)^{T}$ be a re-parameterization of the original parameters such that the log-likelihood function $\ell$ for the observed data can be completely determined by its first $r$ elements, say $\bs{\phi} =  (\omega_1, \dots, \omega_r)^{T}$, through
$$
\ell(\mb{x}, \bs{\phi}) = \sum_{i=1}^{n} \log f(x_i, \bs{\phi}),
$$
where $f(x, \bs{\phi})$ denotes the probability density function for an individual observation $x$. The remaining  $s-r$ parameters of $\bs{\omega}$ are represented by another vector $\bs{\psi} =  (\omega_{r+1}, \dots, \omega_s)^{T}$, which cannot be learned from the observed data. Thus, $\bs{\omega} = (\bs{\phi}, \bs{\psi})$ is partially identified with the identified part $\bs{\phi}$ and the unidentified part $\bs{\psi}$. To further identify $\bs{\psi}$, we make additional assumptions that impose $t$ equality constraints on $\bs{\omega}$:
$$
\mb{h}(\bs{\omega}) = \left( \begin{array}{c} h_1(\bs{\omega}) \\ \vdots \\ h_t(\bs{\omega}) \end{array} \right) = \mb{0}.
$$
To identify $s-r$ unidentified parameters, we need at least $s-r$ equations. Also, it should be reasonable to assume that the number of constraints does not exceed the number of identified parameters, which is necessary for the development of our method. Thus, we assume that $s-r \le t \le r$. Note that the true, though unknown, parameter value $\bs{\omega}^\ast = (\omega_1^\ast, \dots, \omega_s^\ast)^{T}$ is presumed to satisfy these constraints itself, i.e., $\mb{h}(\bs{\omega}^\ast) = \mb{0}$. 

The objective is to find the constrained maximum likelihood estimate $\hat{\bs{\omega}}$ that maximizes the log-likelihood function $\ell(\mb{x}, \bs{\phi})$ subject to the condition $\mb{h}(\hat{\bs{\omega}}) = \mb{0}$. Let $\hat{\bs{\phi}}^{(u)}$ denote the unconstrained maximum likelihood estimate obtained by the general method concerning purely the log-likelihood function $\ell(\mb{x}, \bs{\phi})$. If the equation $\mb{h}( \hat{\bs{\phi}}^{(u)}, \bs{\psi} ) = \mb{0}$ with respect to $\bs{\psi}$ has a solution, say $\hat{\bs{\psi}}^{(c)} $, then $\hat{\bs{\omega}} = (\hat{\bs{\phi}}^{(u)}, \hat{ \bs{\psi}}^{(c)} )$ forms the constrained maximum likelihood estimate of the problem. This approach may fail, however, since the equation $\mb{h}( \bs{\phi}, \bs{\psi} ) = \mb{0}$ with respect to $\bs{\psi}$ may not necessarily have a solution for some values of $\bs{\phi}$. Alternatively, we propose to estimate $\bs{\omega}$ by maximizing $(1/n)\ell(\mb{x},  \bs{\phi}) + \bs{\lambda}^{T} \mb{h} (\bs{\omega})$, where $\bs{\lambda} = (\lambda_1, \dots, \lambda_t)^{T}$ is a Lagrange multiplier.  
Suppose $(\hat{\bs{\omega}}, \hat{\bs{\lambda}})$ solves the following $s+t$ equations:
\begin{align}
\frac{1}{n} \mb{s}(\mb{x},  \bs{\phi} ) + \mb{J}_{\bs{\omega}} \bs{\lambda} &= \mb{0} , \label{eq:lag:e1} \\[4pt] 
\mb{K}_{\bs{\omega}} \bs{\lambda} &= \mb{0} , \label{eq:lag:e2} \\[4pt] 
\mb{h}(\bs{\omega}) &= \mb{0}, \label{eq:lag:e3}
\end{align}
where $\mb{s}(\mb{x},  \bs{\phi} )$ is the score vector of length $r$ whose $i$-th component is $\partial \ell(\mb{x}, \bs{\phi})/\partial \omega_i$, for $i=1, \dots, r$, $\mb{J}_{\bs{\omega}}$ is the $r \times t$ matrix $\left(\partial h_j(\bs{\omega})/\partial \omega_i \right)$, for $i = 1, \dots, r$, $ j=1, \dots, t$,  and $\mb{K}_{\bs{\omega}}$ is the $(s-r) \times t$ matrix $\left( \partial h_j(\bs{\omega})/\partial \omega_{r+i} \right)$, for $i = 1, \dots, s-r$, $ j=1, \dots, t$. Then $\hat{\bs{\omega}}$ should be the constrained maximum likelihood estimate. 

\section{The constrained maximum likelihood estimation}

In this section, we will show that, under some general conditions, if $\mb{x}$ belongs to a set whose probability measure tends to 1 as $n$ approaches infinity, then the equations (\ref{eq:lag:e1}) - (\ref{eq:lag:e3}) have a solution $(\hat{\bs{\omega}}, \hat{\bs{\lambda}})$ such that $\hat{\bs{\omega}}$ is within a small neighborhood of the true value $\bs{\omega}^\ast$. This solution is proved to be the constrained maximum likelihood estimate that maximizes $\ell(\mb{x}, \bs{\phi})$ subject to $\mb{h}(\bs{\omega}) = \mb{0}$. We then extend the definition of $(\hat{\bs{\omega}}, \hat{\bs{\lambda}})$ for all $\mb{x} \in \mathbb{R}^n$, and show the asymptotic distribution of the random variable thus defined. Finally, we propose an algorithm for numerically computing $(\hat{\bs{\omega}}, \hat{\bs{\lambda}})$. The development of this section is based on the work by~\citet{AitchisonSilvey1958}. However, due to the presence of the unidentified component $\bs{\psi}$, our work is more than a simple generalization of their theory. 

We first impose some conditions on $f(x, \bs{\phi})$ and $\mb{h}(\bs{\omega})$ within some neighborhood of $\bs{\omega}^\ast$, say $U_{\alpha} = \left\{ \bs{\omega} : ||\bs{\omega} - \bs{\omega}^\ast|| \leq \alpha \right\}$.  We assume that $f(x, \bs{\phi})$ satisfies the conditions ($\mathcal{F}1$) - ($\mathcal{F}4$) as defined in~\citep{AitchisonSilvey1958}. These conditions are quite general and will be satisfied in most practical estimation problems. Here, we just write one important result implied by these conditions for later reference. If the conditions on $f(x, \bs{\phi})$ are satisfied, for any given positive numbers $\delta < \alpha$ and $\epsilon < 1$ and for sufficiently large $n \geq n(\delta, \epsilon)$, there exists a set $\mb{X}_n$ with the properties
\begin{enumerate}
\item[($\mathcal{X}1$)] $\mathrm{Pr}\{ \mb{X}_n \} > 1 - \epsilon$. 
\item[($\mathcal{X}2$)] $|| \mb{s}(\mb{x}, \bs{\phi}^\ast)/n|| < \delta^2$, if $\mb{x} \in \mb{X}_n$.
\item[($\mathcal{X}3$)] $ \left(\mb{M}_{\mb{x}, \bs{\phi}^\ast}/n\right)$ can be expressed in the form $-\mb{B}_{\bs{\phi}^\ast} + \delta \mb{m}_{\mb{x}, \bs{\phi}^\ast}$, where $\mb{M}_{\mb{x}, \bs{\phi}^\ast}$ is the matrix $(\partial^2 \ell(\mb{x}, \bs{\phi}^\ast)/\partial \omega_i \partial \omega_j)$, $i, j = 1, \dots, r$, $\mb{B}_{\bs{\phi}^\ast}$ is a certain positive definite matrix, and $\mb{m}_{\mb{x}, \bs{\phi}^\ast}$ is an $r \times r$ matrix, the moduli of whose elements are bounded by 1, if $\mb{x} \in \mb{X}_n$.
\item[($\mathcal{X}4$)] For every $\bs{\omega} \in U_{\alpha}$ there exists a constant, say $\kappa_1$, such that 
$$
\left| \frac{1}{n} \frac{\partial^3 \ell(\mb{x}, \bs{\phi})}{\partial \omega_i \partial \omega_j \partial \omega_k} \right| < 2 \kappa_1, 
$$
for all $i,j,k = 1,2,\dots,r$, if $\mb{x} \in \mb{X}_n$.

\end{enumerate}
On the other hand, some conditions are assumed for the constraint function $\mb{h}({\bs{\omega}})$ as follows.
\begin{enumerate}
\item[($\mathcal{H}1$)] For all $\bs{\omega} \in U_\alpha$, the first order partial derivatives $\partial h_k(\bs{\omega})/\partial \omega_i$, $i=1, \dots, s$, $k=1,\dots,t$, exist and they are continuous function of $\bs{\omega}$.
\item[($\mathcal{H}2$)] For all $\bs{\omega} \in U_\alpha$, the second order partial derivatives $\partial^2 h_k(\bs{\omega})/\partial \omega_i \partial \omega_j$, $i,j=1, \dots, s$, $k=1,\dots,t$, exist and $|\partial^2 h_k(\bs{\omega})/\partial \omega_i \partial \omega_j|$ is bounded by a given constant, say $2\kappa_2$, for all $i$, $j$ and $k$.
\item[($\mathcal{H}3$)]  The $r \times t$ matrix $\mb{J}_{\bs{\omega}^\ast}$ and the $(s-r)\times t$ matrix $\mb{K}_{\bs{\omega}^\ast}$ are both of full rank, i.e., $rank(\mb{J}_{\bs{\omega}^\ast}) = t$ and $rank(\mb{K}_{\bs{\omega}^\ast}) =  s-r$. 
\end{enumerate}

\subsection{Existence of the constrained maximum likelihood estimate}

We begin by establishing a necessary and sufficient condition for the existence of a solution of the equations (\ref{eq:lag:e1}) - (\ref{eq:lag:e3}) under some general conditions. It should be noted that the following lemma cannot be directly generalized from the Lemma 1 in~\citep{AitchisonSilvey1958} by simply viewing the log-likelihood function as a function of $\bs{\omega}$ and letting $\mb{B}_{\bs{\omega}^\ast}$ be the $s \times s$ matrix that naturally extends  $\mb{B}_{\bs{\phi}^\ast}$, due to the singularity of $\mb{B}_{\bs{\omega}^\ast}$ thus defined. Therefore, some modifications are required.  

\begin{lemma} \label{thm:l1}
Subject to conditions on $f$ and $\mb{h}$, if $\delta < \alpha$ and $\epsilon <1$ are some given positive numbers and if $\mb{x} \in \mb{X}_n$, then the equations (\ref{eq:lag:e1}) - (\ref{eq:lag:e3}) have a solution $(\hat{\bs{\omega}}$, $\hat{\bs{\lambda}})$ such that $\hat{\bs{\omega}} \in U_\delta$, if and only if $\hat{\bs{\omega}}$ satisfies a certain equation of the form $- \tilde{\mb{B}}_{\bs{\omega}^\ast} (\bs{\omega} - \bs{\omega}^\ast) + \delta^2 \mb{v}(\mb{x}, \bs{\omega}) = \mb{0}$, in which $\tilde{\mb{B}}_{\bs{\omega}^\ast}$ is an $s \times s$ matrix with two blocks on the diagonal being $\mb{B}_{\bs{\phi}^\ast}$ and $\mb{I}_{s-r}$, and $\mb{v}(\mb{x}, \bs{\omega})$ is a continuous function on $U_\delta$ and $||\mb{v}(\mb{x}, \bs{\omega})||$ is bounded for $\omega \in U_\delta$ by a positive number $\kappa^{\dagger}$.
\end{lemma}

\begin{proof}
We first prove the necessity of the condition. By expanding the components of $\mb{s}(\mb{x}, \bs{\phi})$ around $\bs{\phi}^\ast$ in the equation (\ref{eq:lag:e1}), and the components of $\mb{h}(\bs{\omega})$ around  $\bs{\omega}^\ast$ in the equation (\ref{eq:lag:e3}), we find that the solution of the equations (\ref{eq:lag:e1}) - (\ref{eq:lag:e3}) should also satisfy:
\begin{align}
\frac{1}{n}\left\{ \mb{s}(\mb{x}, \bs{\phi}^\ast ) + 
\mb{M}_{\mb{x},\bs{\phi}^\ast} (\bs{\phi} - \bs{\phi}^\ast) + 
\mb{v}^{(1)}(\mb{x}, \bs{\phi})\right\} + \mb{J}_{\bs{\omega}} \bs{\lambda} &= \mb{0},  \label{eq:lemma:e1} \\[4pt]
\mb{J}_{\bs{\omega}^\ast}^{T} (\bs{\phi} - \bs{\phi}^\ast) + \mb{K}_{\bs{\omega}^\ast}^{T}(\bs{\psi} - \bs{\psi}^\ast) + \mb{v}^{(2)}(\bs{\omega}) &= \mb{0}, \label{eq:lemma:e2}
\end{align}
where
\begin{itemize}
\item[(i)] $\mb{v}^{(1)}(\mb{x}, \bs{\phi})$ is a vector of length $r$ whose $m$-th component is 
$$
\frac{1}{2}(\bs{\phi} - \bs{\phi}^\ast)^{T} \mb{L}_{m} (\bs{\phi} - \bs{\phi}^\ast),
$$
where $\mb{L}_{m}$ is the matrix $(\partial^3 \ell(\mb{x}, \bs{\phi}^{(m,1)}) / \partial \omega_m \partial \omega_i \partial \omega_j)$, $i,j=1,\dots, r$, with $\bs{\phi}^{(m,1)}$ being a point such that $||\bs{\phi}^{(m,1)} - \bs{\phi}^\ast|| < ||\bs{\phi} - \bs{\phi}^\ast||$, and
\item[(ii)] $\mb{v}^{(2)}(\bs{\omega})$ is a vector of length $s$ whose $m$-th component is 
$$
\frac{1}{2}(\bs{\omega} - \bs{\omega}^\ast)^{T} \mb{H}_{m} (\bs{\omega} - \bs{\omega}^\ast),
$$
where $\mb{H}_m$ is the matrix $(\partial^2 h_m(\bs{\omega}^{(m,2)}) / \partial \omega_i \partial \omega_j)$, $i, j = 1, \dots, s$, with $\bs{\omega}^{(m,2)}$ being a point such that $||\bs{\omega}^{(m,2)} - \bs{\omega}^\ast|| < ||\bs{\omega} - \bs{\omega}^\ast||$.
\end{itemize}
Further, given property ($\mathcal{X}3$) , we can re-write the equations (\ref{eq:lemma:e1}) and (\ref{eq:lemma:e2}) in the following form:
\begin{align}
- \mb{B}_{\bs{\phi}^\ast} (\bs{\phi} - \bs{\phi}^\ast) +  \mb{J}_{\bs{\omega}} \bs{\lambda} +  \delta^2 \mb{v}^{(3)}(\mb{x}, \bs{\phi}) &= \mb{0}, \label{eq:lemma:e3} \\[4pt]
\mb{J}_{\bs{\omega}^\ast}^{T} (\bs{\phi} - \bs{\phi}^\ast) + \mb{K}_{\bs{\omega}^\ast}^{T}(\bs{\psi} - \bs{\psi}^\ast) + \delta^2 \mb{v}^{(4)}(\bs{\omega}) &= \mb{0}, \label{eq:lemma:e4}
\end{align}
where 
\begin{align}
\mb{v}^{(3)}(\mb{x}, \bs{\phi}) &= \frac{1}{n\delta^2}l(\mb{x}, \bs{\phi}^\ast) + \frac{1}{\delta} \mb{m}_{\mb{x}, \bs{\phi}^\ast} (\bs{\phi} - \bs{\phi}^ \ast) + \frac{1}{n\delta^2} \mb{v}^{(1)}(\mb{x}, \bs{\phi}), \label{eq:lemma:nu3} \\[4pt]
\mb{v}^{(4)}(\bs{\omega}) &= \frac{1}{\delta^2}\mb{v}^{(2)}(\bs{\omega}).  \label{eq:lemma:nu4}
\end{align}
Moreover, by properties ($\mathcal{X}2$) - ($\mathcal{X}4$), we obtain a bound for $\mb{v}^{(3)}(\mb{x}, \bs{\phi})$ as
\begin{align}
||\mb{v}^{(3)}(\mb{x}, \bs{\phi})|| 
&\le \frac{1}{n\delta^2}|| \mb{l}(\mb{x}, \bs{\phi}^\ast)|| + \frac{1}{\delta} ||\mb{m}_{\mb{x}, \bs{\phi}^\ast} (\bs{\phi} - \bs{\phi}^ \ast)|| + \frac{1}{n\delta^2} ||\mb{v}^{(1)}(\mb{x}, \bs{\phi})|| \nonumber \\[4pt]
&< 1 + r^2 + r^3 \kappa_1, \label{eq:lemma:ie1}
\end{align}
and, by condition ($\mathcal{H}2$), we have a bound for $\mb{v}^{(4)}( \bs{\omega})$ as
\begin{align}
||\mb{v}^{(4)}( \bs{\omega})|| 
&< s^3 \kappa_2 \left( \frac{1}{\delta^2} || \bs{\omega} - \bs{\omega}^\ast||  \right) \nonumber \\[4pt]
&<  s^3 \kappa_2. \label{eq:lemma:ie2}
\end{align}

Next, since $\mb{B}_{\bs{\phi}^\ast}$ is positive definite, we can pre-multiply the equation (\ref{eq:lemma:e3}) by $\mb{J}_{\bs{\omega}^\ast}^{T} \mb{B}_{\bs{\phi}^\ast}^{-1}$ to get an expression for $\mb{J}_{\bs{\omega}^\ast}^{T}(\bs{\phi} - \bs{\phi}^\ast)$, which is then plugged into the equation (\ref{eq:lemma:e4}) to obtain the following equation 
\begin{equation}
\mb{J}^{T}_{\bs{\omega}^\ast} \mb{B}^{-1}_{\bs{\phi}^\ast} \mb{J}_{\bs{\omega}} \bs{\lambda} + \mb{K}^{T}_{\bs{\omega}^\ast} (\bs{\psi} - \bs{\psi}^\ast) +  \delta^2 \left( \mb{J}^{T}_{\bs{\omega}^\ast} \mb{B}^{-1}_{\bs{\phi}^\ast} \mb{v}^{(3)} (\mb{x}, \bs{\phi}) + \mb{v}^{(4)}(\bs{\omega}) \right) = \mb{0}. \label{eq:lemma:e5}
\end{equation}
Now the condition ($\mathcal{H}3$) implies that $\mb{J}^{T}_{\bs{\omega}^\ast} \mb{B}^{-1}_{\bs{\phi}^\ast} \mb{J}_{\bs{\omega}^\ast}$ is also positive definite. Besides, according to the condition ($\mathcal{H}1$), the elements of $\mb{J}_{\bs{\omega}}$ are all continuous functions of $\bs{\omega}$. It then follows that $\mb{J}_{\bs{\omega}^\ast}^{T} \mb{B}^{-1}_{\bs{\phi}^\ast} \mb{J}_{\bs{\omega}}$ is also non-singular within $U_\delta$ for sufficiently small $\delta$. Thus, we can solve the equation (\ref{eq:lemma:e5}) with respect to $\bs{\lambda}$ and express it in terms of $\bs{\omega}$
\begin{equation}
\bs{\lambda} = 
- \mb{A}_{\bs{\omega}} \left\{ \mb{K}_{\bs{\omega}^\ast}^{T} (\bs{\psi} - \bs{\psi}^\ast)  + 
 \delta^2 \left( \mb{J}_{\bs{\omega}^\ast}^{T} \mb{B}^{-1}_{\bs{\phi}^\ast} \mb{v}^{(3)} (\mb{x}, \bs{\phi}) + \mb{v}^{(4)}(\bs{\omega}) \right) \right\}, 
\label{eq:lemma:e6}
\end{equation}
where we define the notation $\mb{A}_{\bs{\omega}} = (\mb{J}_{\bs{\omega}^\ast}^{T} \mb{B}^{-1}_{\bs{\phi}^\ast} \mb{J}_{\bs{\omega}} )^{-1}$.

So far, we are basically replicating the steps of the proof given by~\citet{AitchisonSilvey1958}. Now, we need to take some extra steps to find the expression for $(\bs{\psi} - \bs{\psi}^\ast)$. By applying the equation (\ref{eq:lemma:e6}) to substitute for $\bs{\lambda}$, the equation (\ref{eq:lag:e2}) becomes:
\begin{equation}
\mb{K}_{\bs{\omega}} \mb{A}_{\bs{\omega}} \mb{K}_{\bs{\omega}^\ast}^{T} (\bs{\psi} - \bs{\psi}^\ast) +  \delta^2  \mb{K}_{\bs{\omega}} \mb{A}_{\bs{\omega}}
 \left( \mb{J}_{\bs{\omega}^\ast}^{T} \mb{B}^{-1}_{\bs{\phi}^\ast} \mb{v}^{(3)} (\mb{x}, \bs{\phi}) + \mb{v}^{(4)}(\bs{\omega}) \right)
 = \mb{0}. 
\label{eq:lemma:e7}
\end{equation}
Following the same argument for $\mb{J}_{\bs{\omega}^\ast}^{T} \mb{B}^{-1}_{\bs{\phi}^\ast} \mb{J}_{\bs{\omega}}$, the condition ($\mathcal{H}4$) ensures that the matrix $\mb{K}_{\bs{\omega}} \mb{A}_{\bs{\omega}} \mb{K}_{\bs{\omega}^\ast}^{T}$ is not singular within a sufficiently small neighborhood of $\bs{\omega}^\ast$.  
Thus, we can solve the equation (\ref{eq:lemma:e7}) with respect to $\bs{\psi}$ and get
\begin{equation}
\bs{\psi} - \bs{\psi}^\ast = - \delta^2 \mb{v}^{(5)}(\mb{x}, \bs{\omega}), \label{eq:lemma:e8}
\end{equation}
where 
\begin{equation}
\mb{v}^{(5)}(\mb{x}, \bs{\omega}) = 
\left( \mb{K}_{\bs{\omega}} \mb{A}_{\bs{\omega}}  \mb{K}_{\bs{\omega}^\ast}^{T} \right)^{-1}
\left( \mb{K}_{\bs{\omega}} \mb{A}_{\bs{\omega}} \right)
\left( \mb{J}_{\bs{\omega}^\ast}^{T} \mb{B}^{-1}_{\bs{\phi}^\ast}\mb{v}^{(3)}(\mb{x}, \bs{\phi}) + \mb{v}^{(4)}(\bs{\omega}) \right).
\label{eq:lemma:nu5}
\end{equation}
We then plug the equation (\ref{eq:lemma:e8}) into the equation (\ref{eq:lemma:e6}) and derive an updated expression for $\bs{\lambda}$:
\begin{equation}
\bs{\lambda} = - \delta^2
\mb{v}^{(6)}(\mb{x}, \bs{\omega}), 
\label{eq:lemma:e9}
\end{equation} 
where
\begin{equation}
\mb{v}^{(6)}(\mb{x}, \bs{\omega}) = \mb{A}_{\bs{\omega}} \left\{ - \mb{K}_{\bs{\omega}^\ast}^{T} \mb{v}^{(5)}(\mb{x}, \bs{\omega}) + \left( \mb{J}_{\bs{\omega}^\ast}^{T} \mb{B}^{-1}_{\bs{\phi}^\ast}\mb{v}^{(3)}(\mb{x}, \bs{\phi}) + \mb{v}^{(4)}(\bs{\omega}) \right) \right\}.
\label{eq:lemma:nu6}
\end{equation}

By combining the equations (\ref{eq:lemma:e3}) and (\ref{eq:lemma:e8}), with $\bs{\lambda}$ substituted using the equation (\ref{eq:lemma:e9}), we find that the solution of the equations (\ref{eq:lag:e1}) - (\ref{eq:lag:e3}) should also satisfy
\begin{equation}
- \tilde{\mb{B}}_{\bs{\omega}^\ast} (\bs{\omega} - \bs{\omega}^\ast) + \delta^2 \mb{v}(\mb{x}, \bs{\omega}) = \mb{0},
\label{eq:lemma:e10}
\end{equation}
where
$$
\tilde{\mb{B}}_{\bs{\omega}^\ast} = 
\left( \begin{array}{cc} 
\mb{B}_{\bs{\phi}^\ast} & \mb{0} \\[4pt]
\mb{0} & \mb{I}_{s-r} 
\end{array} \right) ,
$$
and
$$
\mb{v}(\mb{x}, \bs{\omega}) = 
\left( \begin{array}{c} 
\mb{v}^{(3)}(\mb{x}, \bs{\phi}) - \mb{J}_{\bs{\omega}}  \mb{v}^{(6)}(\mb{x}, \bs{\omega}) \\[4pt]
- \mb{v}^{(5)}(\mb{x}, \bs{\omega}) 
\end{array} \right).
$$

Finally, we have shown in the inequalities (\ref{eq:lemma:ie1}) and (\ref{eq:lemma:ie2}) that $||\mb{v}^{(3)}(\mb{x}, \bs{\phi})||$ and $|| \mb{v}^{(4)}(\bs{\omega}) ||$ are bounded within $U_\delta$. Also, given that $\mb{A}_{\bs{\omega}}$ and $\mb{K}_{\bs{\omega}} \mb{A}_{\bs{\omega}}  \mb{K}_{\bs{\omega}^\ast}^{T}$ are positive definite within the closed set $U_\delta$, their determinants are both positive within $U_\delta$. Therefore, the continuity of the elements of these two matrices ensures that their determinants are uniformly bounded within $U_\delta$. Then it follows that $\mb{v}(\mb{x}, \bs{\omega})$ is a continuous function on $U_\delta$ and $||\mb{v}(\mb{x}, \bs{\omega})||$ is bounded by a positive number, say $\kappa^{\dagger}$,  for all $\bs{\omega} \in U_\delta$.

Now, we prove the sufficiency of the condition. Suppose the equation (\ref{eq:lemma:e10}) has a solution $\hat{\bs{\omega}}$. That is, $\hat{\bs{\omega}}$ satisfies
\begin{equation}
\left( \begin{array}{cc} 
\mb{B}_{\bs{\phi}^\ast} & \mb{0} \\[4pt]
\mb{0} & \mb{I}_{s-r} 
\end{array} \right)
\left( \begin{array}{c} 
\hat{\bs{\phi}} - \bs{\phi}^\ast \\[4pt]
\hat{\bs{\psi}} - \bs{\psi}^\ast 
\end{array} \right)
= 
\delta^2 \left( \begin{array}{c} 
\mb{v}^{(3)}(\mb{x}, \hat{\bs{\phi}}) - \mb{J}_{\hat{\bs{\omega}}}  \mb{v}^{(6)}(\mb{x}, \hat{\bs{\omega}}) \\[4pt]
- \mb{v}^{(5)}(\mb{x}, \hat{\bs{\omega}}) 
\end{array} \right).
\label{eq:lemma:e11}
\end{equation}
By pre-multiplying the equation (\ref{eq:lemma:e11}) by the $t \times s$ matrix $(\mb{J}_{\bs{\omega}^\ast}^{T} \mb{B}_{\bs{\phi}^\ast}^{-1}, {\ } \mb{K}_{\bs{\omega}^\ast}^{T})$, we have
\begin{equation}
\mb{J}_{\bs{\omega}^\ast}^{T} (\hat{\bs{\phi}} - \bs{\phi}^\ast) + \mb{K}_{\bs{\omega}^\ast}^{T} (\hat{\bs{\psi}} - \bs{\psi}^\ast) + \delta^2 \mb{v}^{(4)}(\hat{\bs{\omega}}) = \mb{0}.
\label{eq:lemma:e12}
\end{equation}
We first write $\mb{v}^{(1)}(\mb{x}, \bs{\phi})$ and $\mb{v}^{(2)}(\bs{\omega})$ as the remainders after expanding $\mb{s}(\mb{x}, \bs{\phi} )$ and $\mb{h}(\bs{\omega})$, respectively, 
\begin{align}
\mb{v}^{(1)}(\mb{x}, \bs{\phi})  &= \mb{s}(\mb{x}, \bs{\phi} ) - \mb{s}(\mb{x}, \bs{\phi}^\ast ) - 
\mb{M}_{\mb{x},\bs{\phi}^\ast} (\bs{\phi} - \bs{\phi}^\ast), \label{eq:lemma:nu1} \\[4pt]
\mb{v}^{(2)}(\bs{\omega}) &= \mb{h}(\bs{\omega})
- \mb{J}_{\bs{\omega}^\ast}^{T} (\bs{\phi} - \bs{\phi}^\ast) - \mb{K}_{\bs{\omega}^\ast}^{T}(\bs{\psi} - \bs{\psi}^\ast). \label{eq:lemma:nu2}
\end{align}
Applying the equations (\ref{eq:lemma:nu1}) and (\ref{eq:lemma:nu2}) to substitute for $\mb{v}^{(1)}(\mb{x}, \bs{\phi})$ and $\mb{v}^{(2)}(\bs{\omega})$ in the equations (\ref{eq:lemma:nu3}) and (\ref{eq:lemma:nu4}), respectively, we get
\begin{align}
\mb{v}^{(3)}(\mb{x}, \bs{\phi}) &= \frac{1}{\delta^2} \left\{ \frac{1}{n} \mb{s} (\mb{x}, \bs{\phi}) + \mb{B}_{\bs{\phi}^\ast}(\bs{\phi} - \bs{\phi}^\ast) \right\}, \label{eq:lemma:nu3:new} \\[4pt]
\mb{v}^{(4)}(\bs{\omega}) &= \frac{1}{\delta^2} \left\{ \mb{h}(\bs{\omega})
- \mb{J}_{\bs{\omega}^\ast}^{T} (\bs{\phi} - \bs{\phi}^\ast) - \mb{K}_{\bs{\omega}^\ast}^{T}(\bs{\psi} - \bs{\psi}^\ast) \right\}.
\label{eq:lemma:nu4:new}
\end{align}
Finally, we substitute for $\mb{v}^{(4)}(\hat{\bs{\omega}})$ in the equation (\ref{eq:lemma:e12}) using the equation (\ref{eq:lemma:nu4:new}). It immediately follows that $\mb{h}(\hat{\bs{\omega}}) = \mb{0}$.

Next, we apply the equations (\ref{eq:lemma:nu3:new}) and (\ref{eq:lemma:nu4:new}) to substitute for $\mb{v}^{(3)}(\mb{x}, \bs{\phi})$ and $\mb{v}^{(4)}(\bs{\omega})$ in the equations (\ref{eq:lemma:nu5}) and (\ref{eq:lemma:nu6}), and end with the following expressions for $ \mb{v}^{(5)}(\mb{x}, \bs{\omega}) $ and $ \mb{v}^{(6)}(\mb{x}, \bs{\omega}) $:
\begin{align}
\mb{v}^{(5)}(\mb{x}, \bs{\omega}) &= - (\bs{\psi} - \bs{\psi}^\ast) + \left( \mb{K}_{\hat{\bs{\omega}}} \mb{A}_{\omega}  \mb{K}_{\bs{\omega}^\ast}^{T} \right)^{-1}  \mb{K}_{\bs{\omega}} \mb{Y}_{\bs{\omega}} \left( \frac{1}{n} \mb{s} (\mb{x}, \bs{\phi}) \right), \label{eq:lemma:nu5:new} \\[4pt]
\mb{v}^{(6)}(\mb{x}, \bs{\omega}) &=  \mb{Y}_{\bs{\omega}} \left( \frac{1}{n} \mb{s} (\mb{x}, \bs{\phi}) \right)  -  \mb{K}_{\bs{\omega}^\ast}^{T} \left( \mb{K}_{\bs{\omega}} \mb{A}_{\omega}  \mb{K}_{\bs{\omega}^\ast}^{T} \right)^{-1}
 \mb{K}_{\bs{\omega}} \mb{Y}_{\bs{\omega}} \left( \frac{1}{n} \mb{s} (\mb{x}, \bs{\phi}) \right), \label{eq:lemma:nu6:new}
\end{align}
where $\mb{Y}_{\bs{\omega}}$ is defined as $ \mb{Y}_{\bs{\omega}} = \mb{A}_{\bs{\omega}} \mb{J}_{\bs{\omega}^\ast}^{T} \mb{B}_{\bs{\phi}^\ast}^{-1} $.  
Now, by using the equations (\ref{eq:lemma:nu3:new}), (\ref{eq:lemma:nu5:new}) and (\ref{eq:lemma:nu6:new}) to substitue for  $\mb{v}^{(3)}(\mb{x}, \hat{\bs{\phi}})$, $\mb{v}^{(5)}(\mb{x}, \hat{\bs{\omega}})$ and $\mb{v}^{(6)}(\mb{x}, \hat{\bs{\omega}})$ in the equation (\ref{eq:lemma:e11}), respectively, we can see that $\hat{\bs{\omega}}$ satisfies
\begin{align*}
\frac{1}{n}\mb{s}(\mb{x}, \hat{\bs{\phi}}) - \mb{J}_{\hat{\bs{\omega}}} \mb{Y}_{\hat{\bs{\omega}}} \left( \frac{1}{n}\mb{s}(\mb{x}, \hat{\bs{\phi}}) \right) &= \mb{0}, \\[4pt]
- \mb{K}_{\hat{\bs{\omega}}} \mb{Y}_{\hat{\bs{\omega}}} \left( \frac{1}{n}\mb{s}(\mb{x}, \hat{\bs{\phi}}) \right)  &= \mb{0}.
\end{align*}
As we have shown earlier that $\mb{h}(\hat{\bs{\omega}}) = \mb{0}$, it is easy to see that $\hat{\bs{\omega}}$, jointly with $\hat{\bs{\lambda}} = - \mb{Y}_{\hat{\bs{\omega}}} \mb{s}(\mb{x}, \hat{\bs{\phi}})/n $, solves the equations (\ref{eq:lag:e1}) - (\ref{eq:lag:e3}).
\end{proof}

We now give the following theorem to show the existence of a solution of the equations (\ref{eq:lag:e1}) - (\ref{eq:lag:e3}). 

\begin{theorem} \label{thm:t1}
Subject to conditions on $f$ and $\mb{h}$, if $\delta$ is a sufficiently small given positive number, $\epsilon$ is a given positive number less than 1 and if $\mb{x} \in \mb{X}_n$, then the equations (\ref{eq:lag:e1}) - (\ref{eq:lag:e3}) have a solution $(\hat{\bs{\omega}}, \hat{\bs{\lambda}})$ such that $\hat{\bs{\omega}} \in U_\delta$.
\end{theorem}

\begin{proof}
The proof of Theorem 1 in ~\citep{AitchisonSilvey1958} works here, provided the modified version of Lemma \ref{thm:l1} given above is used. Also, it is important to notice that the matrix  $\tilde{\mb{B}}_{\bs{\omega}^\ast}$ defined in Lemma 1 is positive definite provided that $\mb{B}_{\bs{\phi}^\ast}$ is positive definite, and its minimum latent root is $\min\{\mu_0, 1\}$, where $\mu_0$ is the latent minimum root of $\mb{B}_{\bs{\phi}^\ast}$. Details are omitted.
\end{proof}

For the remainder of this section, we are going to show that the solution of the equations (\ref{eq:lag:e1}) - (\ref{eq:lag:e3}) as stated in Theorem \ref{thm:t1} locally maximizes the log-likelihood subject to the constraints. This result was proved in~\citep{AitchisonSilvey1958} for the identified model. However, we are not able to prove this result for the partially identified model with a direct extension of their proof. Alternatively, we take another route and use the approach detailed by~\citet{Spring1985}.

To match with the set-up in~\citep{Spring1985}, we change the order of variables and let $\bs{\eta} = (\bs{\lambda}, \bs{\omega})$. Let $\mb{HT}$ denote the second order partial derivatives of  the Lagrangian function $\ell(\mb{x},  \bs{\phi})/n+ \bs{\lambda}^{T}\mb{h}(\bs{\omega})$ evaluated at the critical point $\hat{\bs{\eta}} = ( \hat{\bs{\lambda}}, \hat{\bs{\omega}})$
$$
\mb{HT}^{(n)} = \left(
\begin{array}{ccc}
\mb{0} & \mb{J}_{\hat{\bs{\omega}}}^{T} & \mb{K}_{\hat{\bs{\omega}}}^{T} \\[4pt] 
\mb{J}_{\hat{\bs{\omega}}} & \frac{1}{n}\mb{M}_{\hat{\bs{\phi}}} + \mb{X}_{\hat{\bs{\lambda}}, \hat{\bs{\omega}}} & \mb{Y}_{\hat{\bs{\lambda}}, \hat{\bs{\omega}}}^{T} \\[4pt]
\mb{K}_{\hat{\bs{\omega}}} & \mb{Y}_{\hat{\bs{\lambda}}, \hat{\bs{\omega}}} & \mb{Z}_{\hat{\bs{\lambda}}, \hat{\bs{\omega}}}  
\end{array}
\right), 
$$
where
$$
\left(
\begin{array}{cc}
\mb{X}_{\hat{\bs{\lambda}}, \hat{\bs{\omega}}} &
\mb{Y}_{\hat{\bs{\lambda}}, \hat{\bs{\omega}}}^{T} \\[4pt]
\mb{Y}_{\hat{\bs{\lambda}}, \hat{\bs{\omega}}} &
\mb{Z}_{\hat{\bs{\lambda}}, \hat{\bs{\omega}}} 
\end{array}
\right) = 
\sum_{k=1}^{t}  \hat{\lambda}_{k} 
\left(
\begin{array}{ccc}
\frac{\partial^2 h_k}{\partial \omega_1 \partial \omega_1} & \cdots &
\frac{\partial^2 h_k}{\partial \omega_1 \partial \omega_s} \\[4pt]
\vdots & \ddots & \vdots \\[4pt]
\frac{\partial^2 h_k}{\partial \omega_s \partial \omega_1} & \cdots &
\frac{\partial^2 h_k}{\partial \omega_1s \partial \omega_s} 
\end{array}
\right),
$$
with $\mb{X}_{\hat{\bs{\lambda}}, \hat{\bs{\omega}}}$ being the upper-left $r \times r$ block matrix, $\mb{Y}_{\hat{\bs{\lambda}}, \hat{\bs{\omega}}}$ being the bottom-left $r \times (s-r)$ block matrix, and $\mb{Z}_{\hat{\bs{\lambda}}, \hat{\bs{\omega}}}$ being the bottom-right $(s-r) \times (s-r)$ block matrix. Let $\Lambda^{(n)}_k$ denote the principal upper left $k$-th order minor of the Hessian Matrix $\mb{HT}^{(n)}$. According to Theorem 1 in~\citep{Spring1985}, $\hat{\bs{\omega}}$ locally maximizes the log-likelihood function subject to the constraints, so long as $ (-1)^{t+p} \Lambda^{(n)}_{2t+p}$, $p = 1, \dots,  s-t$, are all positive.

Note that $\hat{\bs{\lambda}}$ was defined as $\hat{\bs{\lambda}} =  -\mb{Y}_{\hat{\bs{\omega}}} \left( \mb{s}(\mb{x}, \hat{\bs{\phi}})/n \right)$. For any small number $\delta$, by the equation (\ref{eq:lemma:nu3:new}) and the inequality (\ref{eq:lemma:ie1}),  if $n$ is sufficiently large, we have
\begin{align*}
\frac{1}{n} || \mb{s} (\mb{x}, \hat{\bs{\phi}}) || &= || - \mb{B}_{\bs{\phi}^\ast}(\hat{\bs{\phi}} - \bs{\phi}^\ast) + \delta^2 \mb{v}^{(3)}(\mb{x}, \hat{\bs{\phi}}) || \\[4pt]
&< \kappa_3 \delta + (1+r^2 + r^3\kappa_1)\delta^2,
\end{align*}
where $\kappa_3$ is a positive number that depends only on the elements of $\mb{B}_{\bs{\phi}^\ast}$. Also, the elements of $\mb{Y}_{\hat{\bs{\omega}}}$ are bounded by a number independent of $\delta$ for $\hat{\bs{\omega}} \in  U_\delta$. Therefore, we have
\begin{align*}
|| \hat{\bs{\lambda}} || &= \frac{1}{n} || \mb{Y}_{\hat{\bs{\omega}}}  \mb{s} (\mb{x}, \hat{\bs{\phi}}) || \\[4pt]
&< \kappa_4 \delta + \kappa_5 \delta^2,
\end{align*}
where $\kappa_4$ and $\kappa_5$ are positive numbers independent of $\delta$. That is, $\hat{\bs{\lambda}}$ converges to $\mb{0}$ as $n$ goes to infinity.  By condition $\mathcal{H}2$, the second partial derivatives $ \partial^2 h_k(\bs{\omega})/\partial \omega_i \partial \omega_j  $, $i, j =1, \dots, s$, $k = 1, \dots, k$, are all bounded by a constant $2\kappa_2$. Thus, it follows that $\mb{X}_{\hat{\bs{\lambda}}, \hat{\bs{\omega}}} \to \mb{0}_{r\times r}$, $\mb{Y}_{\hat{\bs{\lambda}}, \hat{\bs{\omega}}} \to \mb{0}_{r\times (s-r)}$, and $\mb{Z}_{\hat{\bs{\lambda}}, \hat{\bs{\omega}}} \to \mb{0}_{(s-r)\times (s-r)}$. Also, it is easy to see from Theorem 1 that, for $\hat{\bs{\omega}} \in U_\delta$ with sufficiently small value of $\delta$, $\hat{\bs{\omega}}$ converges to $\bs{\omega}^\ast$ as $n$ goes to infinity. By condition ($\mathcal{H}1$), the elements of $\mb{J}_{\bs{\omega}}$ and $\mb{K}_{\bs{\omega}}$ are all continuous functions of $\bs{\omega}$. Thus, as $n$ goes to infinity, $\mb{J}_{\hat{\bs{\omega}}}$, $\mb{K}_{\hat{\bs{\omega}}}$, and $\mb{M}_{\hat{\bs{\phi}}}/n$ approach $\mb{J}_{\bs{\omega}^\ast}$,  $\mb{K}_{\bs{\omega}^\ast}$ and $\mb{M}_{\bs{\phi}^\ast}/n$, respectively. Furthermore, by property ($\mathcal{X}2$), we have $\mb{M}_{\bs{\phi}^\ast}/n$ approaches $-\mb{B}_{\bs{\phi}^\ast}$ as $n$ goes to infinity. Finally, we have $\mb{HT}^{(n)}$ converges to $\mb{HT}^{(\infty)}$ as $n$ goes to infinity, where
$$
\mb{HT}^{(\infty)} = \left(
\begin{array}{ccc}
\mb{0} & \mb{J}_{\bs{\omega}^\ast}^{T} & \mb{K}_{\bs{\omega}^\ast}^{T} \\[4pt]
\mb{J}_{\bs{\omega}^\ast} & -\mb{B}_{\bs{\phi}^\ast}  & \mb{0} \\[4pt]
 \mb{K}_{\bs{\omega}^\ast} & \mb{0} &  \mb{0}
\end{array}
\right) .
$$
Then, for sufficiently large $n$, the signs of the leading principal minors of $\mb{HT}^{(n)}$ are the same as those of their corresponding minors of $\mb{HT}^{(\infty)}$. Therefore, we can instead study the signs of the leading principal minors of $\mb{HT}^{(\infty)}$. 

For brevity, we suppress the subscripts $\bs{\omega}^\ast$ and $\bs{\phi}^\ast$. First, given that $\mb{B}$ is positive definite, by  Sylvester's criterion the upper left $d \times d$ corner matrix of $\mb{B}$, denoted by $\mb{B}_d$, is also positive definite, for $d=1, \dots, r$. Next, since $rank(\mb{J}) = t$, with some re-ordering of the rows if necessary, the first $d$ rows of $\mb{J}$, denoted by $\mb{J}_d$, is a $d \times t$ matrix of full column rank $t$, and thus the matrix $\mb{J}_d^{T} \mb{B}_d^{-1} \mb{J}_d$ is positive definite, for $d  = t+1, \dots, r$. Similarly,  as $rank(\mb{K}) = s-r$, the first $d$ rows of $\mb{K}$, denoted by $\mb{K}_d$, is a $d \times t$ matrix of full row rank $d$, and thus the matrix $\mb{K}_d \left(\mb{J}^{T} \mb{B}^{-1} \mb{J}\right)^{-1}\mb{K}_d^{T}$ is again positive definite, for $d  = 1, \dots, s-r$. Now we are ready to study the sign of $(-1)^{t+p} \Lambda^{(\infty)}_{2t+p}$, for $p = 1, \dots,  s-t$. On one hand, for $p = 1, \dots, r-t$, we have
\begin{align*}
(-1)^{t+p}\Lambda^{(\infty)}_{2t+p} 
&= (-1)^{t+p} \times \mathrm{det} \left( \left(
\begin{array}{ccc}
\mb{0} & \mb{J}_{t+p}^{T} \\[4pt]
 \mb{J}_{t+p}  & -\mb{B}_{t+p}
\end{array}
\right)
\right)\\[4pt]
&=  (-1)^{t+p} \times
\mathrm{det} \left( -\mb{B}_{t+p} \right) \times
\mathrm{det} \left(
 - \mb{J}_{t+p}^{T}  \left(-\mb{B}_{t+p} \right)^{-1} \mb{J}_{t+p}  
\right) \\[4pt]
&=  (-1)^{2t+2p} \times
\mathrm{det} \left( \mb{B}_{t+p} \right) \times
\mathrm{det} \left(
 \mb{J}_{t+p}^{T}  \mb{B}_{t+p}^{-1} \mb{J}_{t+p}  
\right) \\[4pt]
&> 0.
\end{align*}
On the other hand, for $p = r-t+1, \dots, s-t$, we have
\begin{align*}
(-1)^{t+p}\Lambda^{(\infty)}_{2t+p} 
&= (-1)^{t+p} \times
\mathrm{det} \left( \left(
\begin{array}{ccc}
\mb{0} & \mb{J}^{T} & \mb{K}_{t+p - r}^{T} \\[4pt]
\mb{J} & -\mb{B} & \mb{0} \\[4pt]
 \mb{K}_{t+p-r} & \mb{0} & \mb{0}
\end{array}
\right)
\right)\\[4pt]
&= {\ } (-1)^{t+p} \times
\mathrm{det} \left( \left(
\begin{array}{ccc}
\mb{0} & \mb{J}^{T} \\[4pt]
 \mb{J}  & -\mb{B}
\end{array}
\right)
\right) \times \\[4pt]
& \qquad
\mathrm{det} \left( -
\left( \begin{array}{c}
\mb{K}_{t+p-r}^{T} \\[4pt] \mb{0} 
\end{array} \right)
 \left(
\begin{array}{cc}
\mb{0} & \mb{J}^{T}  \\[4pt]
\mb{J} & -\mb{B} 
\end{array}
\right)^{-1}
\left( \begin{array}{cc}
\mb{K}_{t+p-r} & \mb{0}
\end{array} \right)
\right) \\[4pt]
& =  (-1)^{t+p} \times
\mathrm{det} \left( -\mb{B} \right) \times
\mathrm{det} \left(
\mb{J}^{T}  \mb{B}^{-1} \mb{J}  
\right) \times \\[4pt]
& \qquad
\mathrm{det} \left(  
- \mb{K}_{t+p-r} \left( \mb{J}^{T} \mb{B}^{-1} \mb{J}\right)^{-1} \mb{K}_{t+p-r}^{T} \right) \\[4pt]
& =  (-1)^{2t+2p}  \times
\mathrm{det} \left( \mb{B} \right) \times
\mathrm{det} \left(
\mb{J}^{T}  \mb{B}^{-1} \mb{J}  
\right) \times \\[4pt]
& \qquad \mathrm{det} \left(  
\mb{K}_{t+p-r} \left( \mb{J}^{T} \mb{B}^{-1} \mb{J}\right)^{-1} \mb{K}_{t+p-r}^{T} 
\right) \\[4pt]
& > 0.
\end{align*}
Therefore, we have shown that $(-1)^{t+p} \Lambda^{(\infty)}_{2t+p}$, $p=1, \dots, s-t$, is always positive, and so is $(-1)^{t+p} \Lambda^{(n)}_{2t+p}$ for sufficiently large $n$. Thus, it follows that $\hat{\bs{\omega}}$ is the constrained maximum likelihood estimator of the problem. 

\subsection{Asymptotic distributions}

In this section, we define sequences  $\left\{ (\hat{\bs{\omega}}_n, \hat{\bs{\lambda}}_n) \right\}$ that extends $(\hat{\bs{\omega}}, \hat{\bs{\lambda}})$, as stated in the Theorem 1, for all $\mb{x} \in \mathbb{R}^{n}$, and develop the asymptotic distribution for $(\hat{\bs{\omega}}_n, \hat{\bs{\lambda}}_n) $. Note that this section differs from the Section 5 of~\citep{AitchisonSilvey1958} in that the covariance matrix here becomes a partitioned matrix of $3\times 3$ blocks.

\begin{lemma} \label{thm:l2}
The following partitioned matrix is non-singular.
$$
\left( \begin{array}{ccc} 
\mb{B}_{\bs{\phi}^\ast} & \mb{0} & - \mb{J}_{\bs{\omega}^\ast} \\[4pt]
\mb{0} & \mb{0} & - \mb{K}_{\bs{\omega}^\ast} \\[4pt]
- \mb{J}_{\bs{\omega}^\ast}^{T} & - \mb{K}_{\bs{\omega}^\ast}^{T} & \mb{0}
\end{array} \right) 
$$
\end{lemma}
\begin{proof}
For brevity, we omit the suffix $\bs{\phi}^\ast$ and $\bs{\omega}^\ast$. Then we wish to find a matrix 
$$
\left( \begin{array}{ccc}
\mb{P}_{11} & \mb{P}_{12} & \mb{P}_{13} \\[4pt]
\mb{P}_{21} & \mb{P}_{22} & \mb{P}_{23} \\[4pt]
\mb{P}_{31} &\mb{P}_{32} & \mb{P}_{33}
\end{array} \right)
$$
such that
$$
\left( \begin{array}{ccc} 
\mb{B} & \mb{0} & - \mb{J} \\[4pt]
\mb{0} & \mb{0} & - \mb{K} \\[4pt]
- \mb{J}^{T} & - \mb{K}^{T} & \mb{0} 
\end{array} \right)
\left( \begin{array}{ccc} 
\mb{P}_{11} & \mb{P}_{12} & \mb{P}_{13} \\[4pt]
\mb{P}_{21} & \mb{P}_{22} & \mb{P}_{23} \\[4pt]
\mb{P}_{31} & \mb{P}_{32} & \mb{P}_{33}
\end{array} \right)
=
\left( \begin{array}{ccc} 
\mb{I}_{r} & \mb{0} & \mb{0} \\[4pt]
\mb{0} & \mb{I}_{s-r} & \mb{0}  \\[4pt]
\mb{0} & \mb{0} & \mb{I}_{t}
\end{array} \right).
$$
Since $\mb{B}$ is positive definite, and  $\mb{J}$ and $\mb{K}$ are of full rank, it can be solved that 
\begin{align*}
\mb{P}_{11} &= \mb{B}^{-1} - \mb{B}^{-1}\mb{J}(\mb{J}^{T}\mb{B}^{-1}\mb{J})^{-1}\mb{J}^{T} \mb{B}^{-1}  + \\[4pt]
&\qquad \qquad \qquad \qquad \mb{B}^{-1}\mb{J}(\mb{J}^{T}\mb{B}^{-1}\mb{J})^{-1}\mb{K}^{T} \left\{ \mb{K}(\mb{J}^{T}\mb{B}^{-1}\mb{J})^{-1}\mb{K}^{T}  \right\}^{-1} \mb{K}(\mb{J}^{T}\mb{B}^{-1}\mb{J})^{-1}\mb{J}^{T} \mb{B}^{-1}, \\[4pt]
\mb{P}_{12} &= -\mb{B}^{-1}\mb{J}(\mb{J}^{T}\mb{B}^{-1}\mb{J})^{-1}\mb{K}^{T} \left\{ \mb{K}(\mb{J}^{T}\mb{B}^{-1}\mb{J})^{-1}\mb{K}^{T}  \right\}^{-1} ,\\[4pt]
\mb{P}_{13} &=  - \mb{B}^{-1}\mb{J}(\mb{J}^{T}\mb{B}^{-1}\mb{J})^{-1} + \mb{B}^{-1}\mb{J}(\mb{J}^{T}\mb{B}^{-1}\mb{J})^{-1}\mb{K}^{T} \left\{ \mb{K}(\mb{J}^{T}\mb{B}^{-1}\mb{J})^{-1}\mb{K}^{T}  \right\}^{-1} \mb{K}(\mb{J}^{T}\mb{B}^{-1}\mb{J})^{-1}, \\[4pt]
\mb{P}_{22} &= \left\{ \mb{K}(\mb{J}^{T}\mb{B}^{-1}\mb{J})^{-1}\mb{K}^{T}  \right\}^{-1}, \\[4pt]
\mb{P}_{23} &= -\left\{ \mb{K}(\mb{J}^{T}\mb{B}^{-1}\mb{J})^{-1}\mb{K}^{T}  \right\}^{-1} \mb{K}(\mb{J}^{T}\mb{B}^{-1}\mb{J})^{-1}, \\[4pt]
\mb{P}_{33} &= - (\mb{J}^{T}\mb{B}^{-1}\mb{J})^{-1}  + (\mb{J}^{T}\mb{B}^{-1}\mb{J})^{-1}\mb{K}^{T} \left\{ \mb{K}(\mb{J}^{T}\mb{B}^{-1}\mb{J})^{-1}\mb{K}^{T} \right\}^{-1} \mb{K}(\mb{J}^{T}\mb{B}^{-1}\mb{J})^{-1},
\end{align*}
and $\mb{P}_{21}$, $\mb{P}_{31}$, and $\mb{P}_{32}$ are the transposes of $\mb{P}_{12}$, $\mb{P}_{13}$, and $\mb{P}_{23}$, respectively, as it is easy to see that the matrix is symmetric. 
\end{proof}

Suppose $\mb{x} \in \mb{X}_n$, $\delta$ is small enough for Theorem \ref{thm:t1} to apply, and $(\hat{\bs{\omega}}, \hat{\bs{\lambda}})$ is a solution of equations (\ref{eq:lag:e1}) - (\ref{eq:lag:e3}) such that $\hat{\bs{\omega}} \in U_\delta$. We now write the equations (\ref{eq:lag:e1}) - (\ref{eq:lag:e3}) in a different form:
\begin{equation} 
\left(
\begin{array}{ccc} 
\mb{B}_{\bs{\phi}^\ast} +\hat{\mb{b}}(\mb{x}) & \mb{0} & - \mb{J}_{\bs{\omega}^\ast} - \hat{\mb{j}}(\mb{x}) \\[4pt]
\mb{0} & \mb{0} & - \mb{K}_{\bs{\omega}^\ast} - \hat{\mb{k}}(\mb{x}) \\[4pt]
- \mb{J}_{\bs{\omega}^\ast}^{T} - \hat{\mb{j}}'(\mb{x}) & - \mb{K}_{\bs{\omega}^\ast}^{T} - \hat{\mb{k}}'(\mb{x}) & \mb{0} 
\end{array}
 \right)
\left(
\begin{array}{c} 
\hat{\bs{\phi}} - \bs{\phi}^\ast \\[4pt]
\hat{\bs{\psi}} - \bs{\psi}^\ast \\[4pt]
\hat{\bs{\lambda}}
\end{array}
\right)
=
\left(
\begin{array}{c} 
\frac{1}{n} \mb{s}(\mb{x}, \bs{\phi}^\ast) \\[4pt]
\mb{0} \\[4pt]
\mb{0}
\end{array}
\right),
\label{eq:lag:mf1}
\end{equation}
where $\hat{\mb{b}}(\mb{x})$, $\hat{\mb{j}}(\mb{x})$, $\hat{\mb{j}}'(\mb{x})$, $\hat{\mb{k}}(\mb{x})$, and $\hat{\mb{k}}'(\mb{x})$ are matrices whose elements tend to 0 as $\delta$ goes to 0. Thus, by Lemma \ref{thm:l2}, if $\delta$ is sufficiently small, then the matrix 
$$
\left(
\begin{array}{ccc} 
\mb{B}_{\bs{\phi}^\ast} +\hat{\mb{b}}(\mb{x}) & \mb{0} & - \mb{J}_{\bs{\omega}^\ast} -  \hat{\mb{j}}(\mb{x})  \\[4pt]
\mb{0} & \mb{0} & - \mb{K}_{\bs{\omega}^\ast} - \hat{\mb{k}}(\mb{x}) \\[4pt]
- \mb{J}_{\bs{\omega}^\ast}^{T} - \hat{\mb{j}}'(\mb{x}) & - \mb{K}_{\bs{\omega}^\ast}^{T} - \hat{\mb{k}}'(\mb{x}) & \mb{0}
\end{array}
 \right)
$$
is also non-singular and we write its inverse as 
$$
\left( \begin{array}{ccc} 
\hat{\mb{P}}_{11}(\mb{x}) & \hat{\mb{P}}_{12}(\mb{x}) & \hat{\mb{P}}_{13}(\mb{x}) \\[4pt]
\hat{\mb{P}}_{21}(\mb{x}) & \hat{\mb{P}}_{22}(\mb{x}) & \hat{\mb{P}}_{23}(\mb{x}) \\[4pt]
\hat{\mb{P}}_{31}(\mb{x}) &\hat{\mb{P}}_{32}(\mb{x}) & \hat{\mb{P}}_{33}(\mb{x})
\end{array} \right) .
$$
Thus, if $\delta$ is sufficiently small and if $\mb{x} \in \mb{X}_n$, we can solve from the equation (\ref{eq:lag:mf1}) that
\begin{equation}
\left(
\begin{array}{c} 
\hat{\bs{\phi}} - \bs{\phi}^\ast \\[4pt]
\hat{\bs{\psi}} - \bs{\phi}^\ast \\[4pt]
\hat{\bs{\lambda}}
\end{array}
\right)
=
\left( \begin{array}{ccc} 
\hat{\mb{P}}_{11}(\mb{x}) & \hat{\mb{P}}_{12}(\mb{x}) & \hat{\mb{P}}_{13}(\mb{x}) \\[4pt]
\hat{\mb{P}}_{21}(\mb{x}) & \hat{\mb{P}}_{22}(\mb{x}) & \hat{\mb{P}}_{23}(\mb{x}) \\[4pt]
\hat{\mb{P}}_{31}(\mb{x}) &\hat{\mb{P}}_{32}(\mb{x}) & \hat{\mb{P}}_{33}(\mb{x})
\end{array} \right)
\left(
\begin{array}{c} 
\frac{1}{n} \mb{s}(\mb{x}, \bs{\phi}^\ast) \\[4pt]
\mb{0} \\[4pt]
\mb{0}
\end{array}
\right).
\label{eq:lag:mf2}
\end{equation}
Since the asymptotic distribution of $\mb{s}(\mb{x}, \bs{\phi}^\ast)/n$ is known, we can use the above relationship to induce the asymptotic distribution of $(\hat{\bs{\omega}}, \hat{\bs{\lambda}})$. However, this may only be valid for $\mb{x} \in \mb{X}_n$, and we need to extend it to also account for $\mb{x} \notin \mb{X}_n$.

Let $(\delta_m)$, $(\epsilon_m)$ be two decreasing sequences of positive real numbers, such that $\delta_1 < \mu_1/\kappa_3$,  $\epsilon_1 < 1$, and $\delta_m$ and $\epsilon_m$ both tend to 0 as $m$ goes to infinity. Define an increasing sequence ($n_m$) of integers such that, if $n \geq n_m$, there exists a set $\mb{X}_n$ with properties ($\mathcal{X}1$) - ($\mathcal{X}4$) for $\epsilon = \epsilon_m$ and $\delta = \delta_m$. For $m=1,2,\dots,$, if $n_m \leq n < n_{m+1}$, we choose a set $\mb{X}_n$ with properties ($\mathcal{X}1$) - ($\mathcal{X}4$) for $\epsilon = \epsilon_m$ and $\delta = \delta_m$. When $\mb{x} \in \mb{X}_n$, the equations  (\ref{eq:lag:e1}) - (\ref{eq:lag:e3}) have a solution $(\hat{\bs{\omega}}_n, \hat{\bs{\lambda}}_n)$ such that $||\hat{\bs{\omega}}_n - \bs{\omega}^\ast|| < \delta_m$, with  $\hat{\bs{\omega}}_n$ being the constrained maximum likelihood estimate for $\bs{\omega}$. Thus,  $\hat{\bs{\omega}}_n$ and $\hat{\bs{\lambda}}_n$ satisfy the equation (\ref{eq:lag:mf2}). When $\mb{x} \notin \mb{X}_n$, we define
$$
\left(
\begin{array}{c} 
\hat{\bs{\phi}}_n - \bs{\phi}^\ast \\[4pt]
\hat{\bs{\psi}}_n - \bs{\psi}^\ast \\[4pt]
\hat{\bs{\lambda}}_n
\end{array}
\right)
=
\left( \begin{array}{ccc} 
\mb{P}_{11} & \mb{P}_{12} & \mb{P}_{13} \\[4pt]
\mb{P}_{21} & \mb{P}_{22} & \mb{P}_{23} \\[4pt]
\mb{P}_{31} &\mb{P}_{32} & \mb{P}_{33}
\end{array} \right)
\left(
\begin{array}{c} 
\frac{1}{n} \mb{s}(\mb{x}, \bs{\phi}^\ast) \\[4pt]
\mb{0} \\[4pt] 
\mb{0}
\end{array}
 \right),
$$
where $\mb{P}_{ij}$, $i, j = 1, 2, 3$, are defined in the proof of Lemma \ref{thm:l2}. Note that the probability of $\mb{x} \notin \mb{X}_n$ goes to zero as $n$ goes to infinity. Thus, we have defined two sequences of random variables, $(\hat{\bs{\omega}}_n)$ and $(\hat{\bs{\lambda}}_n)$, $n=n_m, n_{m+1},  \dots$, which have the property that $\hat{\bs{\omega}}_n$  converges in probability to $\bs{\omega}^\ast$ as $n$ goes to infinity. Moreover, $\hat{\bs{\omega}}_n$ and $\hat{\bs{\lambda}}_n$ jointly satisfy the equations (\ref{eq:lag:e1}) - (\ref{eq:lag:e3}).

\begin{theorem} \label{thm:t2}
$$\sqrt{n} 
\left( \begin{array}{c}
\hat{\bs{\phi}}_n - \bs{\phi}^\ast \\[4pt]
\hat{\bs{\psi}}_n - \bs{\psi}^\ast \\[4pt]
\hat{\bs{\lambda}}_n
\end{array} \right)  
\stackrel{d}{\rightarrow}   
\mathcal{N}_{s+t} \left( 
\left( \begin{array}{c} 
\mb{0} \\[4pt]
\mb{0} \\[4pt] 
\mb{0}
\end{array} \right), ~
\left( \begin{array}{ccc} 
\mb{P}_{11} & \mb{P}_{12} & \mb{0} \\[4pt]
\mb{P}_{21} & \mb{P}_{22} & \mb{0} \\[4pt]
\mb{0} & \mb{0} & - \mb{P}_{33}
\end{array} \right)
\right).
$$
\end{theorem}
\begin{proof}
 If $\mb{x} \notin \mb{X}_n$, we define $\hat{\mb{P}}_{ij}(\mb{x}) = \mb{P}_{ij}$, $i,j = 1,2,3$. Then, for sufficiently large $n$, we have
$$
\sqrt{n}\left(
\begin{array}{c} 
\hat{\bs{\phi}}_n - \bs{\phi}^\ast \\[4pt]
\hat{\bs{\psi}}_n - \bs{\psi}^\ast \\[4pt]
\hat{\bs{\lambda}}_n
\end{array}
\right)
=
\left( \begin{array}{ccc} 
\hat{\mb{P}}_{11}(\mb{x}) & \hat{\mb{P}}_{12}(\mb{x}) & \hat{\mb{P}}_{13}(\mb{x}) \\[4pt]
\hat{\mb{P}}_{21}(\mb{x}) & \hat{\mb{P}}_{22}(\mb{x}) & \hat{\mb{P}}_{23}(\mb{x}) \\[4pt]
\hat{\mb{P}}_{31}(\mb{x}) &\hat{\mb{P}}_{32}(\mb{x}) & \hat{\mb{P}}_{33}(\mb{x})
\end{array} \right)
\left(  \sqrt{n} \left( 
\begin{array}{c} 
\frac{1}{n} \mb{s}(\mb{x}, \bs{\phi}^\ast) \\[4pt]
\mb{0} \\[4pt]
\mb{0} 
\end{array}
\right) \right).
$$
Since $\hat{\mb{b}}(\mb{x})$, $\hat{\mb{j}}(\mb{x})$, $\hat{\mb{j}}^\ast(\mb{x})$, $\hat{\mb{k}}(\mb{x})$, and $\hat{\mb{k}}^\ast(\mb{x})$ all tend to $\mb{0}$ as $\delta \to 0$, it follows that the elements of 
$$
\left( \begin{array}{ccc} 
\hat{\mb{P}}_{11}(\mb{x}) & \hat{\mb{P}}_{12}(\mb{x}) & \hat{\mb{P}}_{13}(\mb{x}) \\[4pt]
\hat{\mb{P}}_{21}(\mb{x}) & \hat{\mb{P}}_{22}(\mb{x}) & \hat{\mb{P}}_{23}(\mb{x}) \\[4pt]
\hat{\mb{P}}_{31}(\mb{x}) &\hat{\mb{P}}_{32}(\mb{x}) & \hat{\mb{P}}_{33}(\mb{x})
\end{array} \right)
$$
converge in probability to the elements of 
$$
\left( \begin{array}{ccc} 
\mb{P}_{11} & \mb{P}_{12} & \mb{P}_{13} \\[4pt]
\mb{P}_{21} & \mb{P}_{22} & \mb{P}_{23} \\[4pt]
\mb{P}_{31} &\mb{P}_{32} & \mb{P}_{33}
\end{array} \right).
$$
Moreover, it is known that the asymptotic distribution of $\left( \mb{s}(\mb{x}, \bs{\phi}^\ast)/n \right)$ is normal with mean zero and asymptotic variance $\mb{B}_{\bs{\phi}^\ast}$. Thus, we have
$$
\sqrt{n} 
\left( \begin{array}{c} 
\frac{1}{n} \mb{s}(\mb{x}, \bs{\phi}^\ast) \\[4pt]
\mb{0} \\[4pt]
\mb{0} 
\end{array} \right)
\stackrel{d}{\rightarrow} 
\mathcal{N}_{s+t} \left(
\left( \begin{array}{c} 
\mb{0} \\[4pt] 
\mb{0} \\[4pt]
\mb{0}
\end{array} \right),
\left( \begin{array}{ccc} 
\mb{B}_{\bs{\phi}^\ast} & \mb{0} & \mb{0} \\[4pt]
\mb{0} & \mb{0} & \mb{0} \\[4pt]
\mb{0} & \mb{0} & \mb{0}
\end{array} \right)
\right).
$$
It then follows that the asymptotic distribution of $\sqrt{n}
\left( \hat{\bs{\phi}}_n - \bs{\phi}^\ast,  \hat{\bs{\psi}}_n - \bs{\psi}^\ast, 
\hat{\bs{\lambda}}_n \right)$ is
$$
\mathcal{N}_{s+t} \left(
\left( \begin{array}{c} 
\mb{0} \\[4pt]
\mb{0} \\[4pt]
\mb{0}
\end{array} \right),
\left( \begin{array}{ccc} 
\mb{P}_{11} & \mb{P}_{12} & \mb{P}_{13} \\[4pt]
\mb{P}_{21} & \mb{P}_{22} & \mb{P}_{23} \\[4pt]
\mb{P}_{31} &\mb{P}_{32} & \mb{P}_{33}
\end{array} \right)
\left( \begin{array}{ccc} 
\mb{B}_{\bs{\phi}^\ast} & \mb{0} & \mb{0} \\[4pt]
\mb{0} & \mb{0} & \mb{0} \\[4pt]
\mb{0} & \mb{0} & \mb{0}
\end{array} \right)
\left( \begin{array}{ccc} 
\mb{P}_{11} & \mb{P}_{12} & \mb{P}_{13} \\[4pt]
\mb{P}_{21} & \mb{P}_{22} & \mb{P}_{23} \\[4pt]
\mb{P}_{31} &\mb{P}_{32} & \mb{P}_{33}
\end{array} \right)^{T}
\right).
$$
Finally, using the expressions for $\mb{P}_{ij}$, $i, j = 1, 2, 3$, that were derived in the proof of Lemma \ref{thm:l2}, it can be verified that the asymptotic variance is
$$
\left( \begin{array}{ccc} 
\mb{P}_{11} & \mb{P}_{12} & \mb{P}_{13} \\[4pt]
\mb{P}_{21} & \mb{P}_{22} & \mb{P}_{23} \\[4pt]
\mb{P}_{31} &\mb{P}_{32} & \mb{P}_{33}
\end{array} \right)
\left( \begin{array}{ccc} 
\mb{B}_{\bs{\phi}^\ast} & \mb{0} & \mb{0} \\[4pt]
\mb{0} & \mb{0} & \mb{0} \\[4pt]
\mb{0} & \mb{0} & \mb{0}
\end{array} \right)
\left( \begin{array}{ccc} 
\mb{P}_{11} & \mb{P}_{12} & \mb{P}_{13} \\[4pt]
\mb{P}_{21} & \mb{P}_{22} & \mb{P}_{23} \\[4pt]
\mb{P}_{31} &\mb{P}_{32} & \mb{P}_{33}
\end{array} \right)^{T}
=
\left( \begin{array}{ccc} 
\mb{P}_{11} & \mb{P}_{12} & \mb{0} \\[4pt]
\mb{P}_{21} & \mb{P}_{22} & \mb{0} \\[4pt]
\mb{0} & \mb{0} & - \mb{P}_{33}
\end{array} \right).
$$
The result then follows. 
\end{proof}

\subsection{Numerical algorithm}

The solution of the equations (\ref{eq:lag:e1}) - (\ref{eq:lag:e3}), say $\hat{\bs{\omega}} = (\hat{\bs{\phi}}, \hat{\bs{\psi}})$, usually does not have a closed form, and thus must be computed numerically. We may immediately consider the Newton-Raphson method to solve the problem. However, that method requires the form of the Hessian matrix of $\mb{h}(\bs{\omega})$, which is an $s \times s$ matrix and may be very complicated, especially when $s$ is large. Thus, we follow the approach proposed by~\citet{AitchisonSilvey1958} and develop an algorithm that is easier to implement.

Suppose $\bs{\omega}^{(0)} = (\bs{\phi}^{(0)}, \bs{\psi}^{(0)})$ is an initial guess for $\hat{\bs{\omega}}$ such that $||\bs{\omega}^{(0)} - \hat{\bs{\omega}}||$ is small. Then we consider a first order of approximation to $\mb{s}(\mb{x}, \hat{\bs{\phi}})$  and $\mb{h}(\hat{\bs{\omega}})$:
\begin{align*}
\mb{s}(\mb{x}, \hat{\bs{\phi}}) &\approx \mb{s}(\mb{x}, \bs{\phi}^{(0)}) + \mb{M}_{\mb{x}, \bs{\phi}^{(0)}}(\hat{\bs{\phi}} - \bs{\phi}^{(0)}), \\[4pt]
\mb{h}(\hat{\bs{\omega}}) &\approx \mb{h}(\bs{\omega}^{(0)}) + \mb{J}_{\bs{\omega}^{(0)}}^{T}(\hat{\bs{\phi}} - \bs{\phi}^{(0)}) + \mb{K}_{\bs{\omega}^{(0)}}^{T}(\hat{\bs{\psi}} - \bs{\psi}^{(0)}).
\end{align*}
Also, we assume that $\hat{\bs{\lambda}}$ is close to $\mb{0}$ when $n$ is large. Then to a first order of approximation, we have
\begin{align*}
\mb{J}_{\hat{\bs{\omega}}}\hat{\bs{\lambda}} &\approx \mb{J}_{\bs{\omega}^{(0)}}\hat{\bs{\lambda}},  \\[4pt]
\mb{K}_{\hat{\bs{\omega}}}\hat{\bs{\lambda}} & \approx \mb{K}_{\bs{\omega}^{(0)}}\hat{\bs{\lambda}}.
\end{align*}
Since $\hat{\bs{\omega}}$ and $\hat{\bs{\lambda}}$ jointly satisfy the equations (\ref{eq:lag:e1}) - (\ref{eq:lag:e3}), they should also approximately satisfy
$$
\left(
\begin{array}{ccc} 
-\frac{1}{n}\mb{M}_{\mb{x}, \bs{\phi}^{(0)}} & \mb{0} & - \mb{J}_{\bs{\omega}^{(0)}}   \\[4pt]
\mb{0} & \mb{0} & - \mb{K}_{\bs{\omega}^{(0)}} \\[4pt]
- \mb{J}_{\bs{\omega}^{(0)}}^{T}  & - \mb{K}_{\bs{\omega}^{(0)}}^{T}  & \mb{0} 
\end{array}
 \right)
\left(
\begin{array}{c} 
\hat{\bs{\phi}} - \bs{\phi}^{(0)} \\[4pt] 
\hat{\bs{\psi}} - \bs{\psi}^{(0)} \\[4pt]
\hat{\bs{\lambda}}
\end{array}
\right)
\approx
\left(
\begin{array}{c} 
\frac{1}{n} \mb{s}(\mb{x}, \bs{\phi}^{(0)}) \\[4pt] 
\mb{0} \\[4pt]
\mb{h}(\bs{\omega}^{(0)})
\end{array}
 \right).
$$
When $n$ is large, $-\mb{M}_{\mb{x}, \bs{\phi}^{(0)}}/n$ should be close to $\mb{B}_{\bs{\phi}^{(0)}}$. Thus, we use $\mb{B}_{\bs{\phi}^{(0)}}$ to approximate $-\mb{M}_{\mb{x}, \bs{\phi}^{(0)}}/n$. Finally, we have the formula for updating $\bs{\omega}^{(0)}$, and in general for updating $\bs{\omega}^{(r-1)}$ in the $r$-th iteration,
$$
\left( \begin{array}{c}
\bs{\phi}^{(r)} \\[4pt]
\bs{\psi}^{(r)} \\[4pt]
\bs{\lambda}^{(r)}
\end{array} \right)  =
\left( \begin{array}{c} 
\bs{\phi}^{(r-1)} \\[4pt]
\bs{\psi}^{(r-1)} \\[4pt]
\mb{0}
\end{array} \right) +
\left(
\begin{array}{ccc} 
\mb{B}_{\bs{\phi}^{(r-1)}} & \mb{0} & - \mb{J}_{\bs{\omega}^{(r-1)}}   \\[4pt]
\mb{0} & \mb{0} & - \mb{K}_{\bs{\omega}^{(r-1)}} \\[4pt]
- \mb{J}_{\bs{\omega}^{(r-1)}}^{T}  & - \mb{K}_{\bs{\omega}^{(r-1)}}^{T}  & \mb{0} 
\end{array}
\right) ^{-1}
\left( \begin{array}{c} 
\frac{1}{n} \mb{s}(\mb{x}, \bs{\phi}^{(r-1)}) \\[4pt]
\mb{0} \\[4pt]
\mb{h}(\bs{\omega}^{(r-1)})
\end{array}
 \right).
$$
If the sequence $ \left\{ ( \bs{\omega}^{(r)}, \bs{\lambda}^{(r)} ) \right\}$ converges, then it converges to a solution of the equation (\ref{eq:lag:e1}) - (\ref{eq:lag:e3}). Finally, it should be noted that $\bs{\lambda}^{(r-1)}$ is actually missing from the right hand side of the above equation. Thus, the updating procedure only needs to store the current value of $\bs{\omega} = (\bs{\phi}, \bs{\psi})$ for the next iteration. 

\section{Example problem and simulation study} \label{sec:example}

In this section, we use the proposed method to solve a missing data problem, where parameters associated with the missing mechanism may only be identified with additional assumptions. This sort of problem might otherwise be tackled with an expectation-maximization algorithm. More specifically, consider a binary response variable $Y$ and two binary explanatory variables $X_1$ and $X_2$. The probability of having $Y=1$ given $(X_1, X_2)$ is assumed to be determined through a logistic model:
$$
\mathrm{logit} Pr(Y=1|X_1, X_2) = \beta_0 + \beta_1 X_1 + \beta_2 X_2 + \beta_3 X_1 X_2
$$
Suppose we can observe $X_1$ and $X_2$ for everyone sampled, but the status of $Y$ is missing for some people. Let $R$ indicate missingness. The data structure is displayed in Table 1, where $n_{ijk}$ is the number of subjects with complete data of $(Y=i, X_1 = j, X_2 = k, R=1)$, and $m_{jk}$ is the number of subjects with incomplete data of $(X_1 = j, X_2 = k, R=0)$, $i, j, k = 0, 1$. The corresponding cell probabilities, as enclosed in parentheses in the Table \ref{tab:example}, are 
\begin{align*}
r_{ijk} &= Pr(Y=i, X_1 = j, X_2 = k, R=1), \\[4pt]
s_{jk} &= Pr(X_1 = j, X_2 = k, R=0),
\end{align*}
for $i, j, k = 0, 1$. Based on Table \ref{tab:example}, the log-likelihood of data is:
$$
\ell = \sum_{i,j,k} n_{ijk} \log r_{ijk}  + \sum_{j, k} m_{jk} \log s_{jk} .
$$
In order to understand the relationship between $Y$ and $(X_1, X_2)$, we need to infer the proportions of subjects with $Y=1$ among the groups of incomplete data
$$
t_{jk} = Pr(Y=1| X_1 = j, X_2 = k, R=0),
$$
for $j, k = 0, 1$. However, these quantities are not identifiable from data without additional assumptions.
\begin{table}[h]
\begin{center} 
 {\tabcolsep=18pt
 \begin{tabular}{@{}lccc@{}}
 \hline \\[-9pt]
& $Y=0, ~R=1$ & $Y=1, ~R=1$ & $Y=?, ~R=0$ \\ \\[-9pt] 
$X_1=0,~X_2=0, $ & $n_{000}$ ($r_{000}$) & $n_{100}$ ($r_{100}$) & $m_{00}$  ($s_{00}$)\\[4pt]
$X_1=1,~X_2=0, $ & $n_{010}$ ($r_{010}$) & $n_{110}$ ($r_{110}$) & $m_{10}$  ($s_{10}$)\\[4pt]
$X_1=0,~X_2=1, $ & $n_{001}$ ($r_{001}$) & $n_{101}$ ($r_{101}$) & $m_{01}$  ($s_{01}$)\\[4pt]
$X_1=1,~X_2=1, $ & $n_{011}$ ($r_{011}$) & $n_{111}$ ($r_{111}$) & $m_{11}$  ($s_{11}$)\\
\\[-9pt] \hline
\end{tabular}}
\caption{Data structure for the example problem considered in Section \ref{sec:example}. } 
\label{tab:example}
\end{center}
\end{table}

Now, we make two assumptions. First, we assume that the status of $Y$ is missing at random, i.e., $R$ and $Y$ are conditionally independent given $(X_1, X_2)$. This assumption imposes four constraints on parameters, and implies 
$$
\log t_{jk}  - \log ( 1 - t_{jk} )  = \log r_{1jk} - \log r_{0jk},
$$
for $j, k = 0, 1$. Secondly, we assume that the effects of $X_1$ and $X_2$ on $Y$ are additive on the logit scale, which means that the interaction effect $\beta_3$ is zero. This assumption introduces one more constraint on parameters as
$$
\log \frac{\left(r_{100} + s_{00}t_{00}\right)\left(r_{111} + s_{11}t_{11}\right)}{\left(r_{101} + s_{01}t_{01}\right)\left(r_{110} + s_{10}t_{10}\right)}
= 
\log \frac{\left(r_{000} + s_{00}(1-t_{00})\right)\left(r_{011} + s_{11}(1-t_{11})\right)}{\left(r_{001} + s_{01}(1-t_{01})\right)\left(r_{010} + s_{10}(1-t_{10})\right)}.
$$
Under these two assumptions, we can apply the proposed method to obtained the maximum likelihood estimates $\hat{r}_{ijk}$, $\hat{s}_{jk}$, and $\hat{t}_{jk}$, $i, j, k = 0,1$, subject to the above five constraints. Next, the constrained maximum likelihood estimates for the main effects of $X_1$ and $X_2$ can be deduced through
\begin{align*}
\hat{\beta}_1 &= 
\log \frac{\hat{r}_{110} + \hat{s}_{10}\hat{t}_{10}}{\hat{r}_{100} + \hat{s}_{00}\hat{t}_{00}} - 
\log \frac{\hat{r}_{010} + \hat{s}_{10}(1-\hat{t}_{10})}{\hat{r}_{000} + \hat{s}_{00}(1-\hat{t}_{00})}, \\
\hat{\beta}_2 &= 
\log \frac{\hat{r}_{101} + \hat{s}_{01}\hat{t}_{01}}{\hat{r}_{100} + \hat{s}_{00}\hat{t}_{00}} - 
\log \frac{\hat{r}_{001} + \hat{s}_{01}(1-\hat{t}_{01})}{\hat{r}_{000} + \hat{s}_{00}(1-\hat{t}_{00})},
\end{align*}
and the corresponding estimated variances can be obtained by the delta method. 

Finally, based on the above problem, we conduct a simulation study to illustrate the performance of the proposed method.  In particular, we randomly generate 10000 datasets of size 1000 under the parameter setting $\beta_0 = \mathrm{logit}~0.1$, $\beta_1 = \log 2$, $\beta_2 = \log 3$, $\beta_3 = 0$, and
$$
\begin{array}{c}
Pr(X_1=0, X_2=0) = 0.4, \quad Pr(R=0|X_1=0, X_2=0) = 0.2,~~ \\[4pt]
Pr(X_1=1, X_2=0) = 0.3, \quad Pr(R=0|X_1=1, X_2=0) = 0.1,~~ \\[4pt]
Pr(X_1=0, X_2=1) = 0.2, \quad Pr(R=0|X_1=0, X_2=1) = 0.05, \\[4pt]
Pr(X_1=1, X_2=1) = 0.1, \quad Pr(R=0|X_1=1, X_2=1) = 0.05.
\end{array}
$$
For each dataset, we apply the proposed method to obtain the constrained maximum likelihood estimates for $\hat{\beta}_1$ and $\hat{\beta}_2$, and the associated $95\%$ confidence intervals. Our simulation results show that the empirical biases for the estimators of $\beta_1$ and $\beta_2$ are $0.0033$ and $0.0029$, respectively. Correspondingly, the coverage probabilities of the $95\%$ confidence intervals are $95.1\%$ and $95.2\%$, which match well with the nominal level. We can see that the proposed method performs very well. 

%

\section{Just- and over-identified Situations}

In the previous section, we have considered a partially identified model with four non-identifiable parameters and made additional assumptions that impose five constraints on parameters. Consequently, the constrained maximum likelihood estimators for the identifiable parameters,  $r_{ijk}$'s and $s_{jk}$'s, $i, j, k = 0, 1$, differ from their unconstrained estimators. More importantly, comparing to the unconstrained estimators, the constrained estimators are associated with smaller variances. For example, under the parameter setting considered in the previous section, the asymptotic distribution of the the unconstrained estimator for $(r_{000}, \dots, r_{111})$ is 
$$
\left(
\begin{array}{rrrrrrrr}
\mb{0.205} & -0.06 & -0.04 & -0.02 & -0.01 & -0.01 & -0.01 & -0.01 \\[4pt]
-0.06 &  \mb{0.172} & -0.03 & -0.01 & -0.01 & -0.01 & -0.01 & -0.01 \\[4pt]
-0.04 & -0.03 &  \mb{0.122} & -0.01 & -0.01 &  -0.01 &  -0.01 & -0.01 \\[4pt]
-0.02 & -0.01 & -0.01 &  \mb{0.054} &  -0.00 & -0.00 & -0.00 &  -0.00 \\[4pt]
-0.01 & -0.01 & -0.01 &  -0.00 &  \mb{0.031} &  -0.00 &  -0.00 & -0.00 \\[4pt]
-0.01 & -0.01 &  -0.01 & -0.00 &  -0.00 &  \mb{0.047} & -0.00 &  -0.00 \\[4pt]
-0.01 & -0.01 &  -0.01 & -0.00 &  -0.00 & -0.00 &  \mb{0.045} &  -0.00 \\[4pt]
-0.01 & -0.01 & -0.01 &  -0.00 & -0.00 &  -0.00 &  -0.00 &  \mb{0.037}
\end{array}
\right), 
$$
and the asymptotic variance of the corresponding constrained estimator is 
$$
\left(
\begin{array}{rrrrrrrr}
\mb{0.197} & -0.06 & -0.03 & -0.02 & -0.00 & -0.02 & -0.02 & -0.00 \\[4pt]
-0.06 &  \mb{0.165} & -0.04 & -0.01 & -0.02 & -0.00 & -0.00 & -0.02 \\[4pt]
-0.03 & -0.04 &  \mb{0.115} & -0.00 & -0.01 &  0.00 &  0.00 & -0.01 \\[4pt]
-0.02 & -0.01 & -0.00 &  \mb{0.046} &  0.01 & -0.01 & -0.01 &  0.01 \\[4pt]
-0.00 & -0.02 & -0.01 &  0.01 &  \mb{0.023} &  0.01 &  0.01 & -0.01 \\[4pt]
-0.02 & -0.00 &  0.00 & -0.01 &  0.01 &  \mb{0.039} & -0.01 &  0.01 \\[4pt]
-0.02 & -0.00 &  0.00 & -0.01 &  0.01 & -0.01 &  \mb{0.038} &  0.01 \\[4pt]
-0.00 & -0.02 & -0.01 &  0.01 & -0.01 &  0.01 &  0.01 &  \mb{0.029}
\end{array}
\right).
$$
By comparing the elements along the diagonal of these two matrices, it is clear that the constrained estimator is more efficient than the unconstrained estimator for the problem considered in the previous section. 

However, if we only make the missing at random assumption and allow the model for $Y|X_1, X_2$ to be saturated, then we have only four constraints for four non-identifiable parameters. In this case, we find that the constrained and unconstrained maximum likelihood estimators for the identifiable parameters always coincide and have the same asymptotic distribution. Thus, making the missing at random assumption alone leads to no efficiency gain. 

Generally, we say that the parameters are over-identified when the number of constraints is greater than the number of unidentified parameters. In this case, the constrained  maximum likelihood estimator differs from the unconstrained estimator and achieves better efficiency. On the other hand, we say that the parameters are just-identified when the number of constraints is equal to the number of unidentified parameters. If that is the case, the constrained estimator will coincide with the unconstrained estimator, at least asymptotically. Moreover, identifying the unidentified parameters uses up the information provided by the additional constraints and thus an more efficient estimator is not available. This phenomena was also observed by~\citet{ChenChen2011} in the context of a gene-environment independence problem.

\section{Discussion}

Parameters arising from a partially identified model can be estimated when we have enough equality constraints enforced by additional assumptions. The constrained maximum likelihood estimate for the identified part may or may not coincide with its unconstrained counterpart.  When they do not coincide, the constrained version will have a lower estimated variance.

Another possibility for estimating parameters of a partially identified model subject to constraints is to exploit a reduced-form parameterization that is free of constraints. However, the capability of such approach is limited, as a closed form for a reduced-form parameterization is often very complicated or even sometimes not available. In contrast, the method presented in this paper is applicable in more general settings. Moreover, since the log-likelihood function is usually expressed in its simplest form with a transparent re-parameterization, taking the second partial derivatives of the log-likelihood function becomes much more straightforward. Thus, the proposed method is also advantageous in terms of calculation. 

Finally, the proposed method assumes that the partially identified model can be understood through a transparent re-parameterization that separates the identifiable parameters from non-identifiable parameters. Unfortunately, such kind of re-parameterization does not always exist.~\citet{Gustafson2009} gives two examples that do not admit a transparent re-parameterization. In that case, the proposed method may not be applicable. 

\bibliography{reference}
%
%
%
%
%
%
%
%
%
%
%

\end{document}